\documentclass[11pt,a4paper,draft]{amsart}
\usepackage{amsfonts,amsmath,amssymb,amsthm}
\usepackage{times}
\usepackage{color}

\renewcommand{\epsilon}{\varepsilon}

\DeclareSymbolFont{SY}{U}{psy}{m}{n}
\DeclareMathSymbol{\emptyset}{\mathord}{SY}{'306}

\DeclareMathSymbol{\newtimes}{\mathbin}{SY}{'264}

\newcommand{\R}{\mathbb{R}}

{\bf}{\it}

\newtheorem{theorem}{Theorem}{\bf}{\it}
{\bf}{\it}
{\bf}{\it}
{\bf}{\it}
{\it}{\rm}
\newtheorem{lemma}[theorem]{Lemma}{\bf}{\it}
{\it}{\rm}
{\bf}{\it}
{\bf}{\it}
{\bf}{\it}

{\bf}{\it}

\title{An intermediate value theorem in ordered Banach spaces}
\author[V. Kostrykin]{Vadim Kostrykin}
\address{V.~Kostrykin, FB 08 - Institut f\"{u}r Mathematik,
Johannes Gutenberg-Universit\"{a}t Mainz,
Staudinger Weg 9,
D-55099 Mainz,
Germany}
\email{kostrykin@mathematik.uni-mainz.de}

\author[A.~Oleynik]{Anna Oleynik}
\address{A.~Oleynik, FB 08 - Institut f\"{u}r Mathematik,
Johannes Gutenberg-Universit\"{a}t Mainz,
Stau\-din\-ger Weg 9,
D-55099 Mainz,
Germany}
\email{anna.oleynik@inbox.com}



\begin{document}
\maketitle

\begin{abstract}
We consider a monotone increasing operator in an ordered Banach space having $u_-$ and $u_+$ as a strong super- and subsolution, respectively. In contrast with the well studied case $u_+ < u_-$, we suppose that $u_- < u_+$. Under the assumption that the order cone is normal and minihedral, we prove the existence of a fixed point located in the ordered interval $[u_-,u_+].$
\end{abstract}

\footnotetext{This work is supported in part by the Deutsche Forschungsgemeinschaft, grant KO 2936/4-1.}

\bigskip
It is an elementary consequence of the intermediate value theorem for continuous real-valued functions $f:[a_1,a_2]\to\R$ that if either
\begin{equation}\label{eq:1}
f(a_1)>a_1\qquad\text{and}\qquad f(a_2)<a_2
\end{equation}
or
\begin{equation}\label{eq:2}
f(a_1)<a_1\qquad\text{and}\qquad f(a_2)>a_2,
\end{equation}
then $f$ has a fixed point in $[a_1,a_2]$. It is a natural question whether this result can be extended to the case of ordered Banach spaces. A number of fixed point theorems with assumptions of type \eqref{eq:1} are well known, see, e.g.\ \cite[Section 2.1]{Guo}. However, to the best of our knowledge, fixed point theorems with assumptions of type \eqref{eq:2} have not been known so far. In the present note we prove the following fixed point theorem of this type.

\begin{theorem}\label{thm:1}
Let $X$ be a real Banach space with an order cone $K$ satisfying
\begin{itemize}
\item[(a)] $K$ has a nonempty interior,
\item[(b)] $K$ is normal and minihedral.
\end{itemize}
Assume that there are two points in $X$, $u_-\ll u_+$, and a monotone increasing compact continuous operator $T:[u_-,u_+]\to X$. If $u_-$ is a strong supersolution of $T$ and $u_+$ its strong subsolution, that is,
\begin{equation*}
Tu_- \ll u_-\quad\text{and}\quad Tu_+\gg u_+,
\end{equation*}
then $T$ has a fixed point $u_\ast\in[u_-,u_+]$.
\end{theorem}

Theorem \ref{thm:1} generalizes an idea developed by the present authors in \cite{KO}, where the existence of solutions to a certain nonlinear integral equation of Hammerstein type has been shown.

Before we present the proof we recall some notions. We write $u\geq v$ if $u-v\in K$, $u>v$ if $u\geq v$ and $u\neq v$, and $u\gg v$ if $u-v \in \overset{\circ}{K}$, where $\overset{\circ}{K}$ is the interior of the cone $K.$

A cone $K$ is called \emph{minihedral} if for any pair $\{x,y\}$, $x,y\in X$, bounded above in order there exists the least upper bound $\sup\{x,y\}$, that is, an element $z\in X$ such that
\begin{itemize}
\item[(1)] $x\leq z$ and $y\leq z$,
\item[(2)] $x\leq z'$ and $y\leq z'$ implies that $z\leq z'$.
\end{itemize}
Obviously, a cone $K$ is minihedral if and only if for any pair $\{x,y\}$, $x,y\in X$, bounded below in order there exists the greatest lower bound $\inf\{x,y\}$. If a cone has a nonempty interior, then any pair $x,y\in X$ is bounded above in order. Hence, $\sup\{x,y\}$ and $\inf\{x,y\}$ exist for all $x,y\in X$.

A cone $K$ is called \emph{normal} if there exists a constant $N>0$ such that $x\leq y$, $x,y\in K$ implies $\|x\|_X \leq N \|y\|_X$.

By the Kakutani -- Krein brothers  theorem \cite[Theorem 6.6]{Krasnoselski} a real Banach space $X$ with an order cone $K$ satisfying assumptions (a) and (b) of Theorem \ref{thm:1} is isomorphic to the Banach space $C(Q)$ of continuous functions on a compact Hausdorff space $Q$. The image of $K$ under this isomorphism is the cone of nonnegative continuous functions on $Q$.


An operator $T$ acting in the Banach space $X$ is called \emph{monotone increasing} if $u\leq v$ implies $Tu\leq Tv$.

We turn to the proof of Theorem \ref{thm:1}.

Consider the operator $\widehat{T}: [u_-,u_+]\to X$ defined by
\begin{equation}\label{eq:T:hut}
\widehat{T}u := \sup\left\{\inf\{Tu, u_+\}, u_- \right\}.
\end{equation}
Since $\inf\{Tu_+, u_+\}=u_+$ and $\sup\{u_+,u_-\}=u_+$, $u_+$ is a fixed point of the operator $\widehat{T}$. Similarly one shows that $u_-$ is also a fixed point.

\begin{lemma}\label{lem:prop}
The operator $\widehat{T}$ is continuous, monotone increasing, compact, and maps the order interval $[u_-,u_+]$ into itself.
\end{lemma}

\begin{proof}
For any $v\in K$ the maps $u\mapsto \sup\{u,v\}$ and $u\mapsto \inf\{u,v\}$ are continuous, see, e.g., Corollary 3.1.1 in \cite{Chueshov}. Due to the continuity of $T$ it follows immediately that $\widehat{T}$ is continuous as well.
The operator $\widehat{T}$ is monotone increasing since $\inf$ and $\sup$ are monotone increasing with respect to each argument. Therefore, for any $u \in [u_{-},u_{+}]$ we have
  \begin{equation*}
  u_{-}=\widehat{T}u_{-}\leq\widehat{T}u\leq\widehat{T}u_{+}=u_{+}.
  \end{equation*}
Let $(u_n)$ be an arbitrary sequence in $[u_{-},u_{+}]$. Since $T$ is compact, $(Tu_n)$ has a subsequence $(Tu_{n_k})$ converging to some $v\in X$. {}From the continuity of $\widehat{T}$ it follows that the sequence $(\widehat{T}u_{n_k})$ converges to $\sup\left\{\inf\{v, u_+\}, u_- \right\}$, thus, proving that the range of $\widehat{T}$ is relatively compact.
\end{proof}

\begin{lemma}\label{lem:epsilon}
There exist $p_\pm\in X$ with
\begin{equation*}
u_- \ll p_- \ll p_+ \ll u_+
\end{equation*}
and
\begin{equation*}
\widehat{T} p_- < p_-, \qquad \widehat{T} p_+ > p_+.
\end{equation*}
\end{lemma}

\begin{proof}
Due to $T u_- \ll u_-$ there is a $\delta>0$ such that $B_\delta(u_- - Tu_-)\subset\overset{\circ}{K}$. The preimage of $B_\delta(u_- - Tu_-)$ under the continuous mapping $u\mapsto u-Tu$ contains a ball $B_\epsilon(u_-)$. Hence, $u-Tu\gg 0$ holds for all $u\in B_\epsilon(u_-)$. By the same argument $u-Tu\ll 0$ for all $u\in B_\epsilon(u_+)$. Choosing $\epsilon>0$ sufficiently small we can achieve that $B_\epsilon(u_-)\cap B_\epsilon(u_+)=\emptyset$.

Set $p(t):=\{(1-t)u_- + tu_+\, |\, t\in[0,1]\}$. We choose $t_-\in(0,1)$ so small that $p_-:=p(t_-)\in B_\epsilon(u_-)$ and $t_+\in(0,1)$ so close to $1$ that $p_+:=p(t_+)\in B_\epsilon(u_+)$. Then we have $u_-\ll p_-\ll p_+\ll u_+$ and
\begin{equation*}
T p_- \ll p_-,\qquad T p_+ \gg p_+.
\end{equation*}

Due to $Tp_-\ll p_-$ and $p_-\ll u_+$ we have $\inf\{Tp_-, u_+\}= Tp_-$. Further, we obtain
\begin{equation*}
\sup\{T p_-, u_-\} \leq \sup\{p_-,u_-\} = p_-.
\end{equation*}
{}From $Tp_-\ll p_-$ it follows that there is an element $z\ll 0$ such that $Tp_-=p_-+z$. Assume that $\sup\{T p_-, u_-\} = p_-$. Then we have $\sup\{z, u_--p_-\}=0$. However, in view of the Kakutani -- Krein brothers theorem, $u_- - p_-\ll 0$ implies $\sup\{z, u_--p_-\}\ll 0$. Thus, it follows that $\sup\{T p_-, u_-\} \neq p_-$ and, therefore, $\widehat{T} p_- < p_-$. Similarly one shows that $\widehat{T} p_+ > p_+$.
\end{proof}

The main tool for the proof of Theorem \ref{thm:1} is Amann's theorem on three fixed points (see, e.g., \cite[Theorem 7.F and Corollary 7.40]{Zeidler}):

\begin{theorem}\label{thm:1.3}
Let $X$ be a real Banach space with an order cone having a nonempty interior. Assume there are four points in $X$
\begin{equation*}
p_1 \ll p_2 < p_3 \ll p_4
\end{equation*}
and a monotone increasing image compact operator $\widehat{T}:[p_1,p_4]\to X$ such that
 \begin{equation*}
\widehat{T}p_1 = p_1,\quad \widehat{T} p_2 < p_2,\quad \widehat{T} p_3 > p_3,\quad \widehat{T} p_4 = p_4.
\end{equation*}
Then $\widehat{T}$ has a third fixed point $p$ satisfying $p_1<p<p_4$, $p\notin[p_1,p_2]$, and $p\notin[p_3,p_4]$.
\end{theorem}

Recall that the operator is called \emph{image compact} if it is continuous and its image is a relatively compact set.

We choose $p_1=u_-$, $p_2=p_-$, $p_3=p_+$, $p_4=u_+$, where $p_\pm$ as in Lemma \ref{lem:epsilon}. Since the cone $K$ is normal, by Theorem 1.1.1 in \cite{Guo}, the order interval $[u_{-},u_{+}]$  is norm bounded. Thus, the domain of $\widehat{T}$ is a bounded  set. Since $\widehat{T}$ is compact, it is also image compact.

Theorem \ref{thm:1.3} yields the existence of a fixed point $u_\ast$ of operator $\widehat{T}$ satisfying $u_- < u_\ast < u_+$. Obviously, $u_\ast$ is a fixed point of the operator $T$ as well. This observation completes the proof of Theorem \ref{thm:1}.

\subsection*{Acknowledgments}  The authors thank H.-P.~Heinz for useful comments.
 This work has been supported in part by the Deutsche
Forschungsgemeinschaft, grant KO 2936/4-1.

\end{document}